\documentclass[11pt,reqno]{amsart}%
\usepackage[a4paper,margin=3cm]{geometry}%
\usepackage[pagebackref]{hyperref}%
\usepackage[T1]{fontenc}%
\usepackage[utf8]{inputenc}%
\usepackage{lmodern,mathtools}
\usepackage{amsmath,amsfonts,amssymb,amsthm}
\usepackage{amsmath}
\usepackage{tikz}
\usepackage{dsfont}

\title[]{Regularity of solutions of the Dyson equation and applications}

\author{Valentin Pesce}
\address[VP]{CMAP, Ecole Polytechnique, UMR 7641, 91120 Palaiseau, France.}
\email{\url{mailto:valentin.pesce@polytechnique.edu}}

\date{Spring 2026, compiled \today}
\newtheorem{theorem}{Theorem}[section]%

\newtheorem{corollary}[theorem]{Corollary}%
\newtheorem{proposition}[theorem]{Proposition}%
\newtheorem{lemma}[theorem]{Lemma}%
\newtheorem{remark}[theorem]{Remark}%

\let\ensembleNombre\mathbb
\newcommand\tab[1][0.5cm]{\hspace*{#1}}
\newcommand\mes{\mathcal P(\R)}
\newcommand\mesT{\mathcal P(\mathbb T)}
\newcommand\N{\ensembleNombre{N}}
\newcommand\Z{\ensembleNombre{Z}}

\newcommand\R{\ensembleNombre{R}}

\newcommand\T{\ensembleNombre{T}}

\DeclareMathOperator{\leb}{Leb} 
\DeclareMathOperator{\cotan}{cotan} 

\makeatletter
\def\@MRExtract#1 #2!{#1}     
\renewcommand{\MR}[1]{
  \xdef\@MRSTRIP{\@MRExtract#1 !}%
  \href{http://www.ams.org/mathscinet-getitem?mr=\@MRSTRIP}{MR-\@MRSTRIP}}
\makeatother

\numberwithin{equation}{section}


\begin{document}

\begin{abstract}
The goal of this short paper is to investigate the regularity of the solutions of the Dyson equation. In the work of Bertucci and al. \cite{bertucci2022spectral,bertucci2024spectral,bertucci2025new}, a new notion of solutions for the Dyson equation has been introduced using the viscosity solutions theory and they proved a regularization of solutions in $L^\infty$. According to the work of Biane \cite{biane1997free} we should expect a regularization in $C^{1/3}$ of solutions (and not better). In the spirit of the works of Bertucci and al., we shall prove using PDE methods that for almost all time $t>0$ the solution is as expected in $C^{1/3}$. Our approach allows us to extend this result in addition to a drift term in the Dyson equation for which the semi explicit solutions is not necessarily known. We shall also give an application of this result, proving that the solutions of the periodic Dyson equation converge in long-time toward the uniform distribution on the circle in $L^\infty$ norm which was an open question in \cite{bertucci2025new}.
\end{abstract}

\maketitle
 
{\footnotesize\tableofcontents}
\section{Introduction}
In this paper we shall study the $\alpha$-Hölder regularity (in short the $C^\alpha$ regularity) of solutions of the following partial differential equation (in short PDE), known as the Dyson equation:
\begin{equation}
\label{eq:dysonreg}
\partial_t u+\partial_x(uH[u])=0\text{ in } (0,+\infty)\times \R
\end{equation}
where $H[u]$ is the Hilbert transform given by 
$$H[u](x)=\cfrac{1}{\pi}\int_\R\cfrac{u(y)-u(x)}{x-y}\,dy,$$
with the integral understood in the sense of the principal value. 
\newline
The initial condition of \eqref{eq:dysonreg} that shall be considered will be a probability measure $u_0\in\mes$ so that for every time, the solution of \eqref{eq:dysonreg} with initial condition $u_0$ is a probability measure. 
\newline
The PDE \eqref{eq:dysonreg} arises in many different fields as in dislocation dynamics \cite{biler2010nonlinear,forcadel2009homogenization}, fluid mechanics \cite{baker1996analytic,chae2005finite,castro2008global} and also in random matrix theory \cite{dyson1962,anderson2010}. 
\newline
Indeed, recently through the motivation of random matrices there has been a particular interest on \eqref{eq:dysonreg} \cite{lions2022cours,lions2023cours}. Indeed, this equation can be viewed as the limit PDE of the empirical mean of the eigenvalues of the so-called Dyson Brownian motion (see for instance \cite{anderson2010,pesce2025random,Chan,Shi}) and is considered as a "free" analogue of the heat equation in the theory of free probability \cite{voiculescu1993analogues,biane1998stochastic}.
\newline
In \cite{bertucci2022spectral,bertucci2024spectral,bertucci2025new},  the authors use the theory of viscosity solutions to define a robust notion of solutions to \eqref{eq:dysonreg}. Thanks to this approach, the authors gave a more PDE point of view on some results known in the literature of large random matrices mostly through combinatorial arguments or free probability.
\newline
An important result obtained in \cite{bertucci2024spectral,bertucci2025new} is that the Dyson equation regularizes solutions in $L^\infty$ meaning that there exists a universal constant $C>0$ such that for every $u_0\in \mes$ the unique solution $u$ of \eqref{eq:dysonreg} with initial condition $u_0$ has a density for $t>0$ and 
\begin{equation}
\label{eq:Linftyboundregu}
||u(t,.)||_{L^\infty}\le \cfrac{C}{\sqrt{t}}.
\end{equation}
\newline
Having in mind the fundamental solution of the Dyson equation given by the so-called semi-circular distribution we could expect that solutions of the Dyson equation are $C^\alpha$ for every time $t>0$ for a certain $\alpha>0$. 
This result was proved by using some complex analysis techniques from the free probability theory by Biane in \cite{biane1997free}. Indeed, we can express the solution of the Dyson equation as a free sum of the initial condition and the fundamental solution of this PDE. Thanks to this point of view, Biane proved that the solutions of the Dyson equation are at least $C^{1/3}$ for every time $t>0$ (and for a general initial condition we can not expect a better exponent).
\newline 
Obtaining a quantitative form of such result was in particular raised by Pierre-Louis Lions during his course at Collège de France \cite{lions2023cours} and was an open question in the framework of Bertucci et al. \cite{bertucci2022spectral,bertucci2024spectral,bertucci2025new}. 
\newline
On the PDE part of the literature, the $C^{\alpha}$ regularization of solutions of parabolic PDE is reminiscent of the De Giorgi-Nash-Moser technic \cite{de1956sull,nash1958continuity,moser1960new}. This method, that was originally introduced for solving the $\text{XIX}^{th}$ Hilbert's problem for parabolic/elliptic PDE, has been extended to continuity equation with fractional diffusion (as the Dyson equation) in $\R^n$ by Caffarelli et al \cite{caffarelli2013regularity,caffarelli2016regularity} 
. This strategy also proves prove that the solutions of the Dyson equation are in $C^\alpha$ for every time $t>0$ for a certain $\alpha>0$. Nonetheless, as always with the De Giorgi-Nash-Mozer approach the proof is quite technical (for instance the result is not fully proved in dimension 1 in \cite{caffarelli2016regularity} because some additional technical difficulties appear in the case of the dimension 1 but according to the authors it should not changed the strategy and the result obtained for every dimension strictly greater than 1). This motivates to find a simpler PDE approach to prove that the solutions of the Dyson equation have Hölder regularity. Moreover, we shall present an approach that is quite robust for instance if we add a drift term in the Dyson equation for which an explicit solution using free probability is not necessarily known.
\newline
In addition to be a natural question in the scope of PDE theory, the question of regularity of solutions of the Dyson equation has also other applications. For instance, it is also linked with the long-time behaviour of solutions in the periodic case. Indeed, in \cite{bertucci2025new} it was proved that the solution of the periodic Dyson equation converges in all $L^p$ for $1\le p<+\infty$ toward the uniform distribution on the circle when $t$ goes to infinity. In the case $p=+\infty$ the question was not answered since authors lacked some $C^\alpha$ regularity of solutions. Thanks to the new a priori estimates obtained in this paper we shall answer to this question. 
\newline
\newline
\tab In the first part of the paper we establish some a priori estimates that formally imply that solutions of the Dyson equation are in $L^3_t((t,T), C_x^{1/3})$ for every $T\ge t >0$ and so in particular is in $C_x^{1/3}$ for almost all time $t>0$. Then, we explain how to prove that these a priori estimates can be made rigorous for solutions of the Dyson equation in the context of the work of Bertucci et al \cite{bertucci2022spectral,bertucci2024spectral,bertucci2025new}. Finally, we give an application of this result for solutions of the periodic Dyson equation proving that solutions converge in $L^\infty$ toward the uniform distribution on the circle.  
\section{Regularity of solutions of the Dyson equation}
\subsection{Notation }$\,$
\newline
$\bullet$ A function $f:\R\to\R$ is said to be $C^\alpha$ for $0<\alpha<1$ if $\sup_{x\ne y}\frac{|f(x)-f(y)|}{|x-y|^\alpha}<+\infty$. In this case, we shall write $||f||_{C^\alpha}:=\sup_{x\ne y}\frac{|f(x)-f(y)|}{|x-y|^\alpha}$ for the $C^\alpha$ seminorm of $f$.
\newline
$\bullet$ We now recall some basic properties about the Hilbert transform. 
\newline
The Fourier transform of the Hilbert transform is given by $$\mathcal F(H[f])(\xi)=-i\text{ sign}(\xi) \mathcal F(f)(\xi),$$
where $ \mathcal F(g)$ denotes the Fourier transform of $g$ and $\text{sign}$ is the sign function. 
From this we deduce that the Hilbert is an antisymmetric isometry from $L^2(\R)$ to $L^2(\R)$ and for $f \in L^2(\R)$ we have $H[H[f]]=-f$. 
We shall also use the so-called Cotlar's identity for $f\in L^2(\R)$
\begin{equation}
\label{eq:Cotlar identity}
H[f]^2-f^2=2H[fH[f]].
\end{equation}
$\bullet$ We now introduce some standard notation about Sobolev spaces and fractional Laplacian (we give these notations for functions defined on $\R$). 
\newline
For $s\ge 0$, we define the fractional Laplacian $(-\Delta)^s$ as: $$\mathcal F((-\Delta)^{s/2} f)=|.|^s\mathcal F(f).$$ 
In the case $0<s<2$ the fractional Laplacian can be rewritten as an integral operator: 
$$(-\Delta)^s f(x)=c_s\int_\R\cfrac{f(x)-f(y)}{|x-y|^{1+2s}},$$
where $c_s>0$ is an explicit constant.
In particular, we can link $(-\Delta)^{1/2}$ with the Hilbert transform as follows:
$(-\Delta)^{1/2}f=\partial_x H[f]=H[\partial_x f]$.
Using the fractional Laplacian we can as usual define the semi-Sobolev norm for $s>0$
$$||u||_{\dot H^s}^2=\langle(-\Delta)^{s}f,f \rangle_{L^2(\R)}=\|(-\Delta)^{s/2}f \|_{L^2(\R)}^2.$$
\subsection{A priori estimates}
In this section we derive some a priori estimates for smooth solutions of \eqref{eq:dysonreg} that vanishes at infinity. We shall justify in next section how to make rigorous these a priori estimates. Starting from now, we shall use $C$ as a universal constant that does not depend on any parameter.
\newline
To lighten the notation, we write $H[u](t,x)$ instead of $H[u(t,.)](x)$ and $(-\Delta)^s[u](t,x)$ instead of $(-\Delta)^s[u(t,.)](x)$ for a function $u:\R^+\times \R\to \R$.
\begin{lemma}
\label{lemma:entropy}
Let $u$ be a smooth solution to \eqref{eq:dysonreg} that vanishes at infinity. Then $u$ satisfies the following a priori estimate
\begin{equation}
\label{eq:apriorilogol}
\cfrac{d}{dt}\int_\R u(t,x)\log(u(t,x))dx+||u(t,.)||_{\dot{H}^{1/2}}^2= 0.
\end{equation}
\end{lemma}
\begin{proof}
We directly compute by using \eqref{eq:dysonreg} and integrating by parts:
\begin{align*}
\cfrac{d}{dt}\int_\R u(t,x)\log(u(t,x))dx&=\int_\R \partial_t u(t,x)(\log(u(t,x))+1)dx\\
&=-\int_\R \partial_x [uH[u]](t,x)(\log(u(t,x))+1)dx\\
&=\int_\R uH[u](t,x)\cfrac{\partial_x u(t,x)}{u(t,x)}dx\\
&=\int_\R H[u](t,x)\partial_x u(t,x)dx\\
&=-\int_\R (-\Delta)^{1/2}[u](t,x) u(t,x)dx.
\end{align*}
This concludes the proof. 
\end{proof}
This a priori estimate formally implies that a solution of the Dyson equation is in $L^2_t(H^{1/2}_x)$. This does not imply Hölder regularity of solutions to the Dyson equation since $H^{1/2}$ does not embed in any space $C^\alpha$. It almost says that for almost every time $t>0$ the solution is in $L_x^\infty$ which is an information that we already know, as proved in \cite{bertucci2024spectral}. To go further we need another a priori estimate. 
\begin{lemma}
\label{lemma:Cregulapriori}
Let $u$ be a smooth solution of \eqref{eq:dysonreg} that vanishes at infinity. Then $u$ satisfies the following a priori estimate
\begin{equation}
\label{eq:C1/3}
\cfrac{1}{2}\,\cfrac{d}{dt}\, ||u(t,.)||_{\dot{H}^{1/2}}^2+\cfrac{2}{9}\,||u^{3/2}(t,.)||_{\dot H^1}^2+\cfrac{1}{2} \int_\R((-\Delta)^{1/2}[u](t,x))^2u(t,x)dx =0.
\end{equation}
\end{lemma}
\begin{proof}
Using the symmetry of the operator $(-\Delta)^{1/2}$ from $L^2(\R)$ to $L^2(\R)$ we get:
\begin{equation}
\label{eq:identity2}
\begin{split}
\cfrac{1}{2}\,\cfrac{d}{dt}\,||u(t,.)||_{\dot{H}^{1/2}}^2&=\cfrac{1}{2}\cfrac{d}{dt}\int_\R (-\Delta)^{1/2}[u](t,x) u(t,x)dx\\
&=\int_\R\partial_t u(t,x) (-\Delta)^{1/2} [u](t,x) dx\\
&=-\int_\R\partial_x(u(t,x)H[u](t,x))(-\Delta)^{1/2} [u](t,x)dx\\
&=-\int_\R\partial_xu(t,x)H[u](t,x)(-\Delta)^{1/2} [u](t,x)dx-\int_\R u(t,x)((-\Delta)^{1/2} [u](t,x))^2dx.
\end{split}
\end{equation}
We rewrite the first term using the antisymmetry of the Hilbert transform in order to apply the Cotlar identity
\begin{align*}
-\int_\R\partial_x u(t,x)H[u](t,x)(-\Delta)^{1/2} [u](t,x)dx&=\int_\R H[\partial_x u(t,x)(-\Delta)^{1/2} [u](t,x)]u(t,x)dx\\
&=\int_\R H[\partial_x u(t,x)H[\partial_x u](t,x)]u(t,x)dx,
\end{align*}
since $(-\Delta)^{1/2}(f)=H[\partial_x f]$.
\newline 
Applying the Cotlar identity we obtain that this term is equal to 
\begin{equation}
\cfrac{1}{2}\int_\R\left((H[\partial_x u](t,x))^2-(\partial_x u(t,x))^2\right)u(t,x)dx=\cfrac{1}{2}\int_\R\left( ((-\Delta)^{1/2}[u](t,x))^2-(\partial_x u(t,x))^2\right)u(t,x)dx.
\end{equation}
Inserting this equality in \eqref{eq:identity2} yields
\begin{equation}
\begin{split}
\cfrac{1}{2}\,\cfrac{d}{dt}\,||u(t,.)||_{\dot{H}^{1/2}}^2=-\cfrac{1}{2}\int_\R((-\Delta)^{1/2}[u](t,x))^2 u(t,x)dx-\cfrac{1}{2}\int_\R (\partial_x u(t,x))^2 u(t,x)dx
\end{split}
\end{equation}
We conclude by noticing that $$\int_\R (\partial_x u(t,x))^2 u(t,x)dx=\int_\R (\partial_x u(t,x)u^{1/2}(t,x))^2 dx=\cfrac{4}{9}\int_\R (\partial_x (u^{3/2})(t,x))^2 dx=\cfrac{4}{9}\,||u^{3/2}||_{\dot H^1}^2.$$
\end{proof}
This estimates is fundamentally linked with the $C^{1/3}$ regularity expected from the work of Biane since from this estimate we formally obtain that $u^{3/2}$ is in $H^1(\R)$ and so by the Sobolev imbedding $u^{3/2}$ is $C^{1/2}$ and so $u$ is $C^{1/3}$ as the next lemma shows.
\begin{lemma}
\label{lemma:Calpha}
Let $f: \R\to \R^+$ be a real function such that $f^{3/2}$ is $C^{1/2}$ then $f$ is $C^{1/3}$ and more precisely there exists a universal constant $C>0$ such that $||f||_{C^{1/3}}\le C||f^{3/2}||_{C^{1/2}}^{2/3}$. 
\end{lemma}
\begin{proof}
The proof is a straightforward use of the fact $x\to x^{2/3}$ is $2/3$-Hölder: 
\begin{align*}
\sup_{x\ne y}\cfrac{|f(x)-f(y)|}{|x-y|^{1/3}}&= \sup_{x\ne y}\cfrac{|(f^{3/2}(x))^{2/3}-(f^{3/2}(y))^{2/3}|}{|x-y|^{1/3}}\\
&\le C \sup_{x\ne y}\cfrac{|f^{3/2}(x)-f^{3/2}(y)|^{2/3}}{|x-y|^{1/3}}\\
&=C \left(\sup_{x\ne y}\cfrac{|f^{3/2}(x)-f^{3/2}(y)|}{|x-y|^{1/2}}\right)^{2/3}<+\infty.
\end{align*}
\end{proof}
\begin{remark}
It is also natural to ask if we can obtain an a priori estimate on $\frac{d}{dt}||u^{3/2}(t,.)||_{\dot H^1}^2$. After some painful computations, one can obtain:
\begin{equation}
\label{eq:autreapriori}
\cfrac{1}{2}\,\cfrac{d}{dt}||u^{3/2}(t,.)||_{\dot H^1}^2+\cfrac{9}{4}\int_\R(-\Delta)^{1/2}[u]u^2(-\partial_{xx}^2)u = 0.
\end{equation}
This estimate is quite simple but the term $\int_\R(-\Delta)^{1/2}[u]u^2(-\partial_{xx}^2)u$ has no sign and so there is a priori no decrease in time of $||u^{3/2}(t,.)||_{\dot H^1}$.
\end{remark}
\subsection{Regularity of solutions}
\label{section:regularitysolutions}
In the framework of \cite{bertucci2022spectral,bertucci2024spectral} the solutions of \eqref{eq:dysonreg} are in $C([0,T],\mes)$ where $\mes$ is endowed with the topology of weak convergence (the test functions are continuous and bounded). Actually to make rigorous the previous estimates we shall consider the initial condition in $\mathcal P_2(\R)$ the set of probability measures with a finite second moment. A simple computation prove that in this case the solution is in $\mathcal P_2(\R)$ for every time $t>0$ and a simple direct computation show that $$\int_\R u(t,x)x^2 dx=t+\int_\R u(0,x)x^2dx.$$ 
\newline
To justify the a priori estimates, the strategy is usual: we regularize the initial condition and we shall consider a smooth solution to 
\begin{equation}
\label{edpregularisée}
\partial_t u^\varepsilon+\partial_x(u^\varepsilon H[u^\varepsilon])-\varepsilon \partial_{xx}^2u^\varepsilon=0
\end{equation}
for $\varepsilon>0$ and then let $\varepsilon$ goes to 0 using some weak compactness to pass to the limit in the estimates (using that if $x_n$ weakly converges to an $x$ then $||x||\le \liminf ||x_n||$) and the stability of the Dyson equation to obtain inequalities instead of equalities. 
\newline
Let us mention that the existence of a strong solution of \eqref{edpregularisée} with a smooth initial data is standard and was proved in \cite{bertucci2024spectral} (Proposition 4.5) in order to prove the $L^\infty$ regularization of solutions of \eqref{eq:dysonreg}.

\begin{proposition}
Let $u_0\in \mathcal P_2(\R)\cap L^\infty(\R)$. Then the solution to \eqref{eq:dysonreg} with initial condition $u_0$ satisfies the following estimates for $h\ge t> 0$:
\begin{equation}
\label{eq:finalentropy}
\int_\R u(t,x)\log(u(t,x))dx+\int_0^t ||u(s,.)||_{\dot H^{1/2}}^2 ds\le \int_\R u_0(x)\log(u_0(x))dx
\end{equation}
\begin{equation}
\label{eq:final3/2}
||u(h,.)||_{\dot H^{1/2}}^2+\cfrac{4}{9}\int_t^h ||u^{3/2}(s,.)||_{\dot H^{1}}^2 ds\le ||u(t,.)||_{\dot H^{1/2}}^2.
\end{equation}
\end{proposition}
Let us mention that we add an $L^\infty$ assumption on the initial data but actually this is not really an hypothesis since we know from the $L^\infty$ regularization that it is true for every $t>0$. 
\begin{proof}
In a first part of the proof, we explain the main ideas to pass to the limit without going into details which are standard for these kind of estimates. In a second part, I give more explanation about how to rigorously pass to the limit for readers that are less familiar with these type of results. 
\newline
$\bullet$ We first prove the estimate \eqref{eq:finalentropy}. As explained, we consider $u^\varepsilon$ the strong smooth solution to \eqref{edpregularisée} with a smooth initial condition $u_0^\varepsilon:=u_0\ast \rho_\varepsilon$ where $\rho_\varepsilon$ is an approximation of the unity. 
\newline
We obtain by a similar computation as in Lemma \ref{lemma:entropy} that
\begin{equation}
\label{eq:entropylaplacien}
\cfrac{d}{dt}\int_\R u^\varepsilon(t,x)\log(u^\varepsilon(t,x))dx+||u^\varepsilon(t,.)||_{\dot{H}^{1/2}}^2+\varepsilon \int_\R \cfrac{|\nabla u^\varepsilon(t,x)|^2}{u_\varepsilon(t,x)}dx=0.
\end{equation}
We see that the add of the term $-\varepsilon \Delta$ just add an extra term in the a priori estimate. This extra term is known as the Fisher information and is clearly non negative and so we have for all $\varepsilon>0$
\begin{equation}
\label{eq:entropylaplacien2}
\cfrac{d}{dt}\int_\R u^\varepsilon(t,x)\log(u^\varepsilon(t,x))dx+||u^\varepsilon(t,.)||_{\dot{H}^{1/2}}^2\le 0.
\end{equation}
We rewrite this identity as follows for every $t>0$ and $\varepsilon>0$:
\begin{equation}
\label{eq:entropylaplacien3}
\int_\R u^\varepsilon(t,x)\log(u^\varepsilon(t,x))dx+\int_0^t||u^\varepsilon(s,.)||_{\dot{H}^{1/2}}^2ds\le \int_\R u_0^\varepsilon(x)\log(u_0^\varepsilon(x))dx.
\end{equation}
We now just want to pass to the limit. Despite being convex, the quantity $\int_\R f \log f$ is not non negative and so we have to pay attention. I shall follow the argument of Pierre-Louis Lions in his course at Collège de France \cite{lions2023cours}. 
\newline
To overcome this difficulty we use a standard transformation on this entropy term by rewriting it for $f\in\mathcal P_2(\R)$ as $$\int_\R f \log f=\int_\R f\log\left(\frac{f}{\gamma}\right)+\int_\R f\log(\gamma)$$ where $\gamma(x):=\exp(-x^2/2)/\sqrt{2\pi}$ is the density of the standard normal law. Indeed, the quantity $$\int_\R f\log\left(\frac{f}{\gamma}\right)=\int_\R \frac{f}{\gamma}\log\left(\frac{f}{\gamma}\right)\gamma$$ is non negative by the Jensen inequality and is known as the relative entropy of $f\in\mathcal P(\R)$ with respect to $\gamma$. This transformation is possible since we considered an initial condition in $\mathcal P_2(\R)$.  
We rewrite \eqref{eq:entropylaplacien3} with the relative entropy instead of the entropy
\begin{equation}
\label{eq:entropylaplacien4}
\begin{split}
\int_\R u^\varepsilon(t,x)\log\left(\cfrac{u^\varepsilon(t,x)}{\gamma(x)}\right)dx-\cfrac{1}{2}\int_\R& u^\varepsilon(t,x)x^2dx+\int_0^t||u^\varepsilon(s,.)||_{\dot{H}^{1/2}}^2ds\le\\
& \int_\R u^\varepsilon_0(x)\log\left(\cfrac{u^\varepsilon_0(x)}{\gamma(x)}\right)dx-\cfrac{1}{2}\int_\R u^\varepsilon_0(x)x^2dx.
\end{split}
\end{equation}
We can compute the second moment of a solution of \eqref{eq:dysonreg} 
$$\int_\R u^\varepsilon(t,x)x^2dx=(1+2\varepsilon)t+\int_\R u^\varepsilon_0(x)x^2dx.$$
This gives for every $t\ge 0$ and every $\varepsilon>0$ 
\begin{equation}
\label{eq:entropylaplacien5}
\begin{split}
\int_\R u^\varepsilon(t,x)\log\left(\cfrac{u^\varepsilon(t,x)}{\gamma(x)}\right)dx+\int_0^t||u^\varepsilon(s,.)||_{\dot{H}^{1/2}}^2ds\le \int_\R u^\varepsilon_0(x)\log\left(\cfrac{u^\varepsilon_0(x)}{\gamma(x)}\right)dx+\cfrac{1}{2}\,(1+2\varepsilon)t.
\end{split}
\end{equation}
Now we can justify the passage to the limit. Fix an horizon of time $T$ and choose $u_0^\varepsilon$ such that $$\int_\R u^\varepsilon_0(x)\log\left(\cfrac{u^\varepsilon_0(x)}{\gamma(x)}\right)dx \underset{\varepsilon\to 0}{\longrightarrow}  \int_\R u_0(x)\log\left(\cfrac{u_0(x)}{\gamma(x)}\right)dx.$$
Then, the right hand side of \eqref{eq:entropylaplacien5} is bounded uniformly in $\varepsilon>0$ and so deduce that $(u_\varepsilon)_\varepsilon$ is bounded in $L^2([0,T],H^{1/2})$ and so weakly converges up to extraction. 
\newline
Passing to the infimum limit in \eqref{eq:entropylaplacien5} when $\varepsilon$ goes to $0$ yields 
\begin{equation}
\label{eq:entropylaplacien6}
\begin{split}
\int_\R u(t,x)\log\left(\cfrac{u(t,x)}{\gamma(x)}\right)dx+\int_0^t||u(s,.)||_{\dot{H}^{1/2}}^2ds\le \int_\R u_0(x)\log\left(\cfrac{u_0(x)}{\gamma(x)}\right)dx+\cfrac{1}{2}\,t.
\end{split}
\end{equation}
Again using the second moment of a solution of the Dyson equation we obtain the inequality \eqref{eq:finalentropy}.
\newline
\newline
For the second estimate, we follow the same arguments. If we consider a $t>0$ such that $||u(t,.)||_{\dot{H}^{1/2}}<+\infty$. We consider $u^\varepsilon$ as before. Doing the same computation as in Lemma \ref{lemma:Cregulapriori}, we obtain that 
\begin{equation}
\label{eq:C1/3 bis}
\cfrac{1}{2}\,\cfrac{d}{ds}\, ||u^\varepsilon(s,.)||_{\dot{H}^{1/2}}^2+\cfrac{2}{9}\,||(u^{\varepsilon})^{3/2}(s,.)||_{\dot H^1}^2+\cfrac{1}{2} \int_\R((-\Delta)^{1/2}[u](s,x))^2u(s,x)dx +\varepsilon ||u^\varepsilon(s,.)||^2_{\dot H^{3/2}}=0.
\end{equation}
In particular, this gives for every $\varepsilon>0$ and $h\ge t$: 
\begin{equation}
\label{eq:C1/3 bisbis}
\cfrac{1}{2}\, ||u^\varepsilon(h,.)||_{\dot{H}^{1/2}}^2+\cfrac{2}{9}\,\int_t^h||(u^{\varepsilon})^{3/2}(s,.)||_{\dot H^1}^2ds\le \cfrac{1}{2}\, ||u^\varepsilon(t,.)||_{\dot{H}^{1/2}}^2.
\end{equation}
Hence $(u^\varepsilon(h,.))_{h\ge t}$ is bounded in $H^{1/2}$ and  $(u^{\varepsilon})^{3/2}$ in $L^2(t,h,H^1)$. So we can extract a subsequence that weakly converges in these spaces and we can pass to the infimum limit to obtain the result. In particular this implies that $||u^\varepsilon(h,.)||_{\dot{H}^{1/2}}$ is finite for every $h\ge t$ if $u(t,.)\in H_x^{1/2}$ and since by the first estimate $||u^\varepsilon(t,.)||_{\dot{H}^{1/2}}$ is finite for almost all $t>0$ we can conclude. 
\newline
\newline 
$\bullet$ We now give some complements to justify the limits. Indeed, from the notion of solutions introduced in \cite{bertucci2022spectral,bertucci2024spectral,bertucci2025new}, a priori we just know that $u^\varepsilon$ converges toward $u$ the solution of the Dyson equation (in the viscosity solutions sense) in $C([0,T],\mes)$ where $\mes$ is endowed with the topology of weak convergence. 
\newline
To make this convergence stronger, we can notice from the first estimate that $(u^\varepsilon)_\varepsilon$ is bounded in $L^2(0,T,H^{1/2})$. Moreover, $u^\varepsilon$ solves $$\partial_t u^\varepsilon+\partial_x(u^\varepsilon H[u^\varepsilon])-\varepsilon \Delta u_\varepsilon=0$$ and we know that $(u^\varepsilon)_\varepsilon$ is bounded in all $L^p$ so we deduce that $u^\varepsilon H[u^\varepsilon]$ is bounded in $L^2$ and so $\partial_x(u^\varepsilon H[u^\varepsilon])$ is bounded in $H^{-1}$. Hence, we get that $(u^\varepsilon)_\varepsilon$ is bounded in $L^2(0,T,H^{-2})$. A standard application of the Aubin-Lions-Simon lemma gives that up to extraction $(u^\varepsilon)_\varepsilon$ converges strongly in $L^2(0,T,L_{\text{loc}}^2)$ and so again up to a diagonal extraction procedure almost everywhere. This makes rigorous the passage to the limit using the Fatou Lemma in the entropy term in \eqref{eq:finalentropy}. To rigorously conclude this inequality we have to justify that the weak limit of $(u^\varepsilon)_\varepsilon$ in $L^2(0,T,H^{1/2})$ is $u$, the solution of the Dyson equation. We call $w$ this weak limit point. In particular $(u^\varepsilon)_\varepsilon$ weakly converges toward $w$ in $L^2(0,T,L^2)$. We recall that more generally in an Hilbert space $\mathcal H$ if $(x_n)_{n\in\N}\in \mathcal H^\N$ is a bounded sequence of $\mathcal H$ and $x\in\mathcal H$ are such that $\langle x_n,z\rangle\underset{n\to+\infty}{\longrightarrow}\langle x,z\rangle$ for all $z$ in a dense subset of $\mathcal H$, then $(x_n)_{n\in\N}$ weakly converges toward $x$. If we test $u_\varepsilon$ against a $C^\infty$ compactly supported function, it converges toward $u$ tested against this test function by the convergence in $C([0,T],\mes)$. This complete the first inequality. 
\newline
For the second one, the main problem is again to justify that if we again denote $w$ the weak limit point of $((u^\varepsilon)^{3/2})_\varepsilon$ in $L^2(t,h,H^1)$, we have $w=u^{3/2}$ where $u$ is the solution of the Dyson equation. Again $((u^\varepsilon)^{3/2})_\varepsilon$ weakly converges toward $w$ in $L^2(t,h,L^2)$ and we have that $u^{3/2}$ is in $L^2(t,h,L^2)$ since $u$ is bounded in time in $L_x^3$. Moreover, since we have for every $y,z\in\R^+$ that $$|y^{3/2}-z^{3/2}|\le C \sqrt{\max(y,z)}|y-z|$$  we obtain that for every compact $K$ $$\int_{t}^{h}\int_K \left|(u^{\varepsilon})^{3/2}-u^{3/2}\right|^{2}\le C \int_{t}^{h}\int_K \left|u^{\varepsilon}-u\right|^{2}$$ using the $L^\infty$ bounds \eqref{eq:Linftyboundregu} of $u$ and $u_\varepsilon$ that are independent of $\varepsilon$ and the initial condition and of time on the interval $(t,h)$. 
From the Aubin-Lions-Simon lemma stated below we deduce that for all compact $K$ $$\int_{t}^{h}\int_K \left|(u^{\varepsilon})^{3/2}-u^{3/2}\right|^{2}\underset{\varepsilon\to 0}{\longrightarrow}0.$$
From this and the Cauchy-Schwarz inequality, we deduce that$(u^{\varepsilon})^{3/2}$ weakly converges toward $u^{3/2}$ in $L^2(t,h,L^2)$ and this makes rigorous the convergence for the second inequality.  
\end{proof}

\begin{corollary}
Let $u_0\in\mathcal P_2(\R)$ and $u$ be the solution of the Dyson equation with initial condition $u_0$ then for every $t>0$ $u(t,.)\in H^{1/2}$ and $$\int_{t}^{+\infty}||u^{3/2}(s,.)||_{\dot H^{1}}^2 ds<\infty.$$
\end{corollary}
\begin{proof}
Since for every $t>0$ the solution of the Dyson equation is in $L_x^\infty$, we can apply Proposition \ref{prop:estiméespreuves} for $t>0$. The first inequality implies that $u(t,.)\in H_x^{1/2}$ for almost all $t>0$ and the second one implies that $t\to ||u(t,.)||^2_{\dot H^{1/2}}$ is non-increasing on $(0,+\infty)$. So, for all $t>0$ $u(t,.)$ is in $H_x^{1/2}$. From the second estimate, we also deduce that for all $t>0$ $$\int_{t}^{+\infty}||u^{3/2}(s,.)||_{\dot H^{1}}^2 ds\le ||u(t,.)||^2_{\dot H^{1/2}}<+\infty.$$
\end{proof}
As an immediate consequence of the Sobolev embedding of $H^1(\R)$ in $C^{1/2}(\R)$ and Lemma \ref{lemma:Calpha} we obtain the following result
\begin{corollary}
Let $u_0\in \mathcal P_2(\R)$ then the solution of \eqref{eq:dysonreg} with initial condition $u_0$ satisfies that for every $T\ge t>0$:
$$\int_t^{T}||u(s,.)||_{C^{1/3}}^3 ds <+\infty.$$ In particular, for almost all $s>0$, $u(s,.)\in C^{1/3}$.
\end{corollary}

\subsection{Extension in the case of the Dyson equation with a drift}
We can obtain a similar result in the case of the add of a drift $b$ with some standard assumptions on $b$.
We call  Dyson equation with a drift $b$ the following PDE:
\begin{equation}
\label{eq:dysonwithadriftregu}
\partial_t u+\partial_x(u(H[u]+b))=0.
\end{equation}
A typical example in which such PDEs appear is when instead of just considering in random matrices the Dyson Brownian motion, we consider standard Ornstein-Uhlenbeck processes instead of Brownian motions. In this case, the drift is $b(x)=x=V'(x)$ is associated to a confining potential $V(x)=x^2/2$.
\begin{proposition}
Assume that $b\in L^\infty_t(W^{1,+\infty}_x)$ then $u$; the solution of the Dyson equation with drift $b$ and initial condition $u_0\in \mathcal P_2(\R)$;
satisfies the following estimates for every $h\ge t>0$: 
\begin{equation}
\label{eq:finalentropydrift}
\int_\R u(h,x)\log(u(h,x))dx+\int_t^h ||u(s,.)||_{\dot H^{1/2}}^2 ds\le \int_\R u_t(x)\log(u(t,x))dx+(h-t)\sup_{s\ge 0}\|\partial_x b(s,.)\|_\infty
\end{equation}
\begin{equation}
\label{eq:final3/2drift}
\begin{split}
\cfrac{1}{2}\,||u(h,.)||_{\dot H^{1/2}}^2+\cfrac{2}{9}\int_t^h ||u^{3/2}(s,.)||_{\dot H^{1}}^2 ds\le \cfrac{1}{2}\,||u(t,.)&||_{\dot H^{1/2}}^2+ \sup_{s\ge 0}||\partial_x b(s,.)||_{\infty}^2(h-t)+\\
&2\sup_{s\ge 0}||\partial_x b(s,.)||_\infty \int_t^h||u(s,.)||_{\dot H^{1/2}}^2ds.
\end{split}
\end{equation}
\end{proposition}
\begin{proof}
We focus on a priori estimates that can be made rigorous as in Proposition \ref{prop:estiméespreuves}.
\newline
$\bullet$ For the first estimate, doing exactly the same computations as in Lemma \ref{lemma:entropy}, we obtain 
\begin{align*}
\cfrac{d}{dt}\int_\R u(t,x)\log(u(t,x))dx+||u(t,.)||_{\dot{H}^{1/2}}^2&=\int_\R\partial_x u(t,x)b(t,x)dx\\
&=-\int_\R u(t,x)\partial_x b(t,x)dx\\
&\le \sup_{s\ge 0}\|\partial_x b(s,.)\|_{\infty}||u(t,.)||_{L^1}=\sup_{s\ge 0}\|\partial_x b(s,.)\|_\infty.
\end{align*}
\newline
$\bullet$ Again doing similar computations as in Lemma \ref{lemma:Cregulapriori} we obtain: 
\begin{equation}
\label{eq:C1/3drift}
\cfrac{1}{2}\,\cfrac{d}{dt}\, ||u(t,.)||_{\dot{H}^{1/2}}^2+\cfrac{2}{9}\,||u^{3/2}(t,.)||_{\dot H^1}^2+\cfrac{1}{2} \int_\R((-\Delta)^{1/2}[u](t,x))^2u(t,x)dx =-\int_\R\partial_x(u b)(-\Delta)^{1/2}(u)dx.
\end{equation}
We now have to bound the right hand side coming from the add of the drift. The techniques that shall be used are standard for the estimation of $H^s$ norms of solutions of continuity equations.
\newline 
We compute 
$$-\int_\R\partial_x(u b)(-\Delta)^{1/2}(u)dx=-\int_\R\partial_x u b (-\Delta)^{1/2}(u)dx-\int_\R u \partial_x b(-\Delta)^{1/2}(u)dx=:I_1+I_2.$$
We start to bound $I_2$ using that $xy\le \frac{\varepsilon^2}{2} x^2+\frac{1}{2\varepsilon^2} y^2$: 
\begin{align*}
I_2&\le \cfrac{\varepsilon^2}{2}\int_\R (\partial_x b)^2 u dx+\cfrac{1}{2\varepsilon^2}\int_\R ((-\Delta)^{1/2}(u))^2u\\
&\le \cfrac{\varepsilon^2\sup_{s\ge 0}||\partial_x b(s,.)||_{\infty}^2}{2}\int_\R  u dx+\cfrac{1}{2\varepsilon^2}\int_\R ((-\Delta)^{1/2}(u))^2u\\
&=\cfrac{\varepsilon^2\sup_{s\ge 0}||\partial_x b(s,.)||_{\infty}^2}{2}+\cfrac{1}{2\varepsilon^2}\int_\R ((-\Delta)^{1/2}(u))^2u,
\end{align*}
since $u(t,.)\in\mes$ for every $t\ge 0$.
\newline
Choosing $\varepsilon=\sqrt{2}$ and inserting the bound of $I_2$ in \eqref{eq:C1/3drift} we get 
\begin{equation}
\label{eq:C1/3driftbis}
\cfrac{1}{2}\,\cfrac{d}{dt}\, ||u(t,.)||_{\dot{H}^{1/2}}^2+\cfrac{2}{9}\,||u^{3/2}(t,.)||_{\dot H^1}^2+\cfrac{1}{4} \int_\R((-\Delta)^{1/2}[u](t,x))^2u(t,x)dx \le\sup_{s\ge 0}||\partial_x b(s,.)||_{\infty}^2+I_1.
\end{equation}
It remains to bound $I_1$. To do so, we shall use some standard commutator estimates called Kato-Ponce commutator estimates (or sometimes Kato-Ponce-Vega-Kenig or even fractional Leibnitz rule see for instance \cite{li2019kato}).
First, we rewrite $I_1$ using a commutator:

\begin{align*}
I_1&=-\int_\R(-\Delta)^{1/4}[\partial_x u b] (-\Delta)^{1/4}(u)dx\\
&=-\int_\R[(-\Delta)^{1/4},b](\partial_x u) (-\Delta)^{1/4}(u)dx-\int_\R b(-\Delta)^{1/4}[\partial_x u] (-\Delta)^{1/4}(u)dx\\
&=:I_1^1+I_1^2,
\end{align*}
where more generally $[(-\Delta)^s,f](g):=(-\Delta)^s(fg)-f(-\Delta)^s(g)$ for every functions $f$ and $g$.
We bound $I_1^2$: 

\begin{align*}
I_1^2&=-\cfrac{1}{2}\int_\R b\partial_x(((-\Delta)^{1/4}(u))^2)dx\\
&=\cfrac{1}{2}\int_\R \partial_x b((-\Delta)^{1/4}(u))^2dx
\le \sup_{s\ge 0}||\partial_x b(s,.)||_\infty \int_\R ((-\Delta)^{1/4}(u))^2dx \\
&= \sup_{s\ge 0}||\partial_x b(s,.)||_\infty ||u(t,.)||_{\dot H^{1/2}}^2
\end{align*}
For $I_1^1$, the Cauchy-Schwarz inequality and Kato-Ponce commutator estimates yields:
\begin{align*}
|I_1^1|&\le ||(-\Delta)^{1/4}(u)||_{L^2}||(-\Delta)^{1/4},b](\partial_x u)||_{L^2}\\
&\le ||u(t,.)||_{\dot H^{1/2}}\sup_{s\ge 0}||\partial_x b(s,.)||_\infty|| \,||(-\Delta)^{\frac{1}{4}-1}(\partial_x u)||_{L^2}\\
&=||u(t,.)||_{\dot H^{1/2}} \sup_{s\ge 0}||\partial_x b(s,.)||_\infty|| \,||u(t,.)||_{\dot H^{1/2}}\le\sup_{s\ge 0}||\partial_x b(s,.)||_\infty \,||u(t,.)||_{\dot H^{1/2}}^2
\end{align*}
Hence, inserting these bounds in \eqref{eq:C1/3driftbis}, we get
\begin{equation}
\label{eq:C1/3driftbisbis}
\begin{split}
\cfrac{1}{2}\,\cfrac{d}{dt}\, ||u(t,.)||_{\dot{H}^{1/2}}^2+\cfrac{2}{9}\,||u^{3/2}(t,.)||_{\dot H^1}^2&+\cfrac{1}{4} \int_\R((-\Delta)^{1/2}[u](t,x))^2u(t,x)dx \\&\le \sup_{s\ge 0}||\partial_x b(s,.)||_{\infty}^2+2\sup_{s\ge 0}||\partial_x b(s,.)||_\infty \,||u(t,.)||_{\dot H^{1/2}}^2.
\end{split}
\end{equation}
Let us mention that in this case, we can not say that the $H^{1/2}$ norm of a solution of \eqref{eq:dysonwithadriftregu} is non-increasing with respect to time but since 
\begin{equation}
\label{eq:C1/3driftbisbisbis}
\cfrac{1}{2}\,\cfrac{d}{dt}\, ||u(t,.)||_{\dot{H}^{1/2}}^2 \le \sup_{s\ge 0}||\partial_x b(s,.)||_{\infty}^2+2\sup_{s\ge 0}||\partial_x b(s,.)||_\infty \,||u(t,.)||_{\dot H^{1/2}}^2,
\end{equation}
the Grönwall lemma implies that for all $h\ge t$ 
$$ ||u(h,.)||_{\dot{H}^{1/2}}^2 \le \left( ||u(t,.)||_{\dot{H}^{1/2}}^2 +\frac{\sup_{s\ge 0}||\partial_x b(s,.)||_{\infty}}{2}\right)e^{4\sup_{s\ge 0}||\partial_x b(s,.)||_{\infty}(h-t)}-\cfrac{\sup_{s\ge 0}||\partial_x b(s,.)||_{\infty}}{2}.$$
In particular if $u(t,.)\in H_x^{1/2}$ then for every time $h\ge t$ we also have $u(h,.)\in H_x^{1/2}$.
\newline
Moreover, integrating \eqref{eq:C1/3driftbisbis} from $t$ to $h$ gives \eqref{eq:final3/2drift}.
\end{proof}
\section{Application to the long-time behaviour of solutions of the periodic Dyson equation}
In this section, we give an application of the regularization proved in the previous part for the case of the periodic Dyson equation:
\begin{equation}
\label{eq:dysoncirculairedp}
\partial_t u+\partial_x(uH[u])=0 \text{ in } (0,+\infty)\times \T,
\end{equation}
where $\T=\R/2\pi\Z$ is the circle and $$H[u](x)=\frac{1}{2\pi}\int_\T \cotan\left(\frac{x-y}{2}\right)(u(y)-u(x))dy$$ is the periodic Hilbert transform. Again we look at solutions of \eqref{eq:dysoncirculairedp} with initial condition in $\mesT$ so that the solution is in $\mesT$ for every time.
\newline
The notion of viscosity solutions for this PDE was studied in \cite{bertucci2025new} and a $L^\infty$ regularization was also proved, as in the real case. It was proved that the maximum of solutions of \eqref{eq:dysoncirculairedp} converges toward $\frac{1}{2\pi}$ and that solutions converge in all $L^p$ for $1\le p<+\infty$ toward the uniform distribution on the circle. The convergence in $L^\infty$ norm of solutions of \eqref{eq:dysoncirculairedp} toward the uniform distribution was an open question in \cite{bertucci2025new}.
\newline
Doing exactly the same computations and arguments, the results of Section \ref{section:regularitysolutions} can also obtained  for solutions of \eqref{eq:dysoncirculairedp}.
\newline
Hence, we can without loss of generality suppose that the solutions of the Dyson equation are in $L^\infty$ and satisfy:
\begin{equation}
\label{Dysonhyp}
\int_0^{+\infty}||u^{3/2}(s,.)||_{\dot H^1}^2 ds <+\infty.
\end{equation} 
This has two consequences. First, for almost all $t>0$
$||u^{3/2}(t,.)||_{\dot H^1}^2<+\infty$ and so $u^{3/2}(t,.)\in H^1$ for almost all $t>0$. In particular, the Sobolev embedding implies that for almost all $t>0$ $u^{3/2}(t,.)$ is $C^{1/2}$ and so $u(t,.)$ is $C^{1/3}$ for almost all $t>0$ by Lemma \ref{lemma:Calpha}. 
\newline 
Secondly \eqref{Dysonhyp} implies that we can find $K>0$ and a sequence of time $t_n\underset{n\to+\infty}{\longrightarrow}+\infty$ such that $||u^{3/2}(t_n,.)||_{\dot H^1}^2\le K$ for all $n\ge 0$.
Again by the Sobolev embedding and Lemma \ref{lemma:Calpha} there exists a constant (still denoted $K$) such that for all $n\ge 0$
\begin{equation}
\label{eq:contradictionavenir}
||u(t_n,.)||_{C^{1/3}}\le K.
\end{equation} 
\begin{proposition}
\label{prop:estiméespreuves}
Let $u$ be a solution of the periodic Dyson equation that satisfies \eqref{eq:contradictionavenir}, then there exists a time $t>0$ such that $\min_{x\in\T}u(t,x)>0$. 
\end{proposition}
\begin{proof}
The idea of the proof is quite simple. Since the mass is preserved and the maximum goes to $1/2\pi$ when the time goes to $+\infty$ we have that there is more and more mass near $1/2\pi$. If the minimum of the solution is 0 for every time, we would have for large $t$ some points which are very near a point of minima of $u$ and whose value is very near $1/2\pi$. This would imply that the $C^\alpha$ norm of the solution would explode but using the assumption \eqref{eq:contradictionavenir} this will not be possible. 
\newline
Let us formalize this idea. In this proof, we denote $M(t):=\max_\T u(t,.)$. Assume by contradiction that for every time $t>0$ the minimum of $u(t,.)$ is zero. Let $x(t)\in\T$ be such that $u(t,x(t))=0$. Fix $\varepsilon>0$ that has to be thought as small and consider for every $t>0$ $$J(\varepsilon,t)=(x(t)-\varepsilon,x(t)+\varepsilon).$$ We also fix $\varepsilon'>0$ and we define for every $t>0$ $$I(\varepsilon',t)=\{y\in\T\,|\, u(t,y)<\frac{1}{2\pi}-\varepsilon'\}.$$
In \cite{bertucci2025new} it was proved in Proposition 5.16 (just using the conservation of the mass of a solution of \eqref{eq:dysoncirculairedp}) that $$\leb(I(\varepsilon',t))\le 2\pi \cfrac{M(t)-\frac{1}{2\pi}}{\varepsilon'+M(t)-\frac{1}{2\pi}}.$$
Indeed, we have: 
\begin{align*}
0&=1-\int_\T \mu(t,\theta)d\theta=\int_\T\left(\cfrac{1}{2\pi}-\mu(t,\theta)\right)d\theta\\
&=\int_{I(\varepsilon',t)}\left(\cfrac{1}{2\pi}-\mu(t,\theta)\right)d\theta+\int_{\T-I(\varepsilon',t)}\left(\cfrac{1}{2\pi}-\mu(t,\theta)\right)d\theta\\
&\ge\varepsilon' \leb(I(\varepsilon',t))+\left(\cfrac{1}{2\pi}-M(t)\right)\left(2\pi-\leb(I(\varepsilon',t))\right).
\end{align*}
Since $M(t)\underset{t\to +\infty}{\longrightarrow} \frac{1}{2\pi}$ we obtain that for all $\varepsilon'$, $\leb(I(\varepsilon',t))\underset{t\to +\infty}{\longrightarrow} 0$.
So we can consider $T(\varepsilon,\varepsilon')$ such that $\leb(I(\varepsilon',t))<2\varepsilon$ for every $t\ge T(\varepsilon,\varepsilon')$.
\newline
For $t>T(\varepsilon,\varepsilon')$ we have: $$I(\varepsilon',t)^c\cap J(\varepsilon,t)\ne \emptyset,$$
where $A^c$ is the complementary of the set $A$.
Indeed, if it was not the case $J(\varepsilon,t)$ would be included in $I(\varepsilon',t)$ which is not possible since the Lebesgue measure of $J(\varepsilon,t)$ is $2\varepsilon$ and the Lebesgue measure of $I(\varepsilon',t)$ is strictly smaller than $2\varepsilon$ by definition of $T(\varepsilon,\varepsilon')$.
\newline
So we can consider $y(\varepsilon,\varepsilon',t)\in I(\varepsilon',t)^c\cap J(\varepsilon,t)$. 
\newline
By definition of the $C^{1/3}$ norm, for $t\ge T(\varepsilon,\varepsilon')$, we get 
\begin{align*}
||u(t,.)||_{C^{1/3}}&\ge \cfrac{|u(t,x(t))-u(y(\varepsilon,\varepsilon',t))|}{|x(t)-y(\varepsilon,\varepsilon',t)|^{1/3}}\\
&\ge \cfrac{\frac{1}{2\pi}-\varepsilon'}{\varepsilon^{1/3}}
\end{align*}
Consider $\varepsilon>0$ such that the right side is strictly greater than $K$. This is in contradiction with $\eqref{eq:contradictionavenir}$ in $t=t_n$ for $n$ large enough.
\end{proof}
Finally, it was proved in Proposition 5.20 in \cite{bertucci2025new} that if the minimum of a solution of \eqref{eq:dysoncirculairedp} is strictly positive at a certain time then the solution converges toward the uniform distribution on the circle in $L^\infty$ norm. Combining these results give the convergence in $L^\infty$ norm of solutions of \eqref{eq:dysoncirculairedp} toward the uniform distribution.
\begin{corollary}
Let $u$ be the solution of the periodic Dyson equation with initial condition $u_0\in\mesT$. Then, there exists a time $t>0$ such that $\min_{x\in\T} u(t,x)>0$ which implies $$\|u(t,.)-\frac{1}{2\pi}\|_{\infty}\underset{t\to +\infty}{\longrightarrow} 0.$$
\end{corollary}

\subsection*{Acknowledgments} The author is thankful to Charles Bertucci (Ceremade) for the discussions on the question and to Pierre-Louis Lions (Collège de France) for pointing out the problem and for his course at Collège de France on the topic.
\renewcommand{\MR}[1]{}
\bibliographystyle{smfplain}
\bibliography{ref}

\end{document}